\documentclass[a4paper]{article}
\usepackage[left=3.5cm,right=3.0cm,top=2.5cm,bottom=2.5cm]{geometry}
\usepackage{xcolor}
\usepackage{graphicx}
\usepackage{cite}
\usepackage{bm}
\usepackage{amsmath,amssymb,amscd,amsxtra,amsfonts}
\usepackage{lipsum}
\usepackage{amsfonts}
\usepackage{graphicx}
\usepackage{multirow}

\usepackage{array}
\usepackage{bm}
\usepackage{amsmath}
\usepackage{amssymb}
\usepackage{verbatim}
\usepackage{stmaryrd}
\usepackage{xcolor}

\newcommand\keywordsname{Key words}
\newcommand\AMSname{AMS subject classifications}

\newenvironment{@abssec}[1]
{\if@twocolumn
\section*{#1}%
\else
\vspace{.05in}\footnotesize
\parindent .2in
{\upshape\bfseries #1. }\ignorespaces
\fi}

{\if@twocolumn\else\par\vspace{.1in}\fi}

\newenvironment{keywords}{\begin{@abssec}{\keywordsname}}{\end{@abssec}}

\providecommand*{\Dist}[2]{\operatorname{dist}({#1};{#2})}   
\providecommand*{\Dist}[2]{\Dist{#1}{#2}}

\newcommand{\be}{\begin{eqnarray}}
\newcommand{\ee}{\end{eqnarray}}

\newcommand{\ben}{\begin{eqnarray*}}
\newcommand{\een}{\end{eqnarray*}}

\newtheorem{theorem}{Theorem}[section]
\newtheorem{lemma}[theorem]{Lemma}%
\newtheorem{remark}{Remark}%

\numberwithin{equation}{section}
\newtheorem{proof}{\textbf{Proof.}}

\title{Optimal error estimate of an isoparametric upwind discontinuous Galerkin method for radiation transport equation on curved domains}
\author{Changhui Yao \thanks{C. Yao. School of Mathematics and Statistics, Zhengzhou University, 450001, Zhengzhou, China. (chyao@lsec.cc.ac.cn)}
\and
Yunpan Ma
\thanks{Y. Ma. School of Mathematics and Statistics, Zhengzhou University, 450001, Zhengzhou, China. (yunpanma@gs.zzu.edu.cn)}
\and
Lingxiao Li
\thanks{L. Li (\textbf{Corresponding Author}). School of Mathematics and Statistics, Henan University, 475004, Kaifeng, China. (lilingxiao@lsec.cc.ac.cn)}}
\begin{document}
\date{2026-05-06}
\maketitle

\begin{abstract}
 In recent years, high-order finite element methods on high-order meshes have attracted considerable attention. This work investigates the isoparametric upwind discontinuous Galerkin method for the radiation transport equation on a bounded domain with a piecewise $C^{k+1}$ smooth curved boundary. We use the isoparametric mapping to approximate the curved domain and construct a curved upwind discontinuous Galerkin scheme.
The first-order hyperbolic nature and the complexity introduced by non-affine transformation, lead to additional difficulties for geometric approximation, numerical stability and the optimal error estimate.
To address these issues, with the help of an isoparametric auxiliary operator, we first prove that the bilinear form is continuous with respect to the DG norm when its first argument is the isoparametric projection error.
Then the geometric approximation error of inflow boundary of original domain is precisely estimated.
The error order between discrete normal vectors and the continuous ones are also proven.
Finally, the rigorous analysis yields an optimal convergence rate in the DG norm.
Two- and three-dimensional numerical tests are conducted to support the theoretical results.
\end{abstract}

\begin{keywords}
Radiation  transport equation; Upwind discontinuous Galerkin method; Auxiliary mapping; Isoparametric auxiliary operator; Optimal convergence rate
\end{keywords}

\section{Introduction}
\label{intro}

The radiation transport equation (RTE) describes the propagation of photons in scattering media and their interactions with matter, with significant applications in fields such as astrophysics, atmospheric science, and high energy density physics \cite{LEHTIKANGAS2015,MODEST2021,Tang2021,Fu2025}.
A simplified form of the stationary RTE can be derived  for angular discretization by employing the discrete ordinates (SN) method \cite{Lathrop1966,Badri2018,Zong2023}. In this paper, we consider the governing equation
\begin{equation}
\Omega \cdot \nabla I + \sigma I = f \ \  in \ \   D, \quad I = g \ \  on \ \  \Gamma^{-},\label{eq1.1}
\end{equation}
where $I$ is the radiation intensity, $\sigma$ is a bounded absorption coefficient, $\Omega$ is a constant unit discrete direction. $D$ is a bounded curved domain with boundary $\Gamma=\partial D$, $\Gamma^{-} = \{x \in \Gamma: \Omega \cdot \boldsymbol{n}(x) < 0 \}$ is the inflow part of the boundary $\Gamma$, and $ \boldsymbol{n}(x) $ is the unit outward normal vector at $x\in\Gamma$.

When applying high-order finite elements to problems with curved boundaries, the geometric error induced by the usual affine and polynomial meshes generally prevents us from achieving optimal convergence rates, which are typically attainable in the polygonal or polyhedral domain. Consequently, accurate approximation of the boundaries becomes essential. One of the most popular strategies in scientific computing and engineering is the isoparametric finite elements \cite{LENOIR1986,LI2023,AYLWIN2023}, wherein the discrete mesh uses shape functions of the same order as the spatial discretization. This approach has garnered considerable attention, with extensive research devoted to its application across diverse problem
settings \cite{CIARLET1972,LENOIR1986,Cheng2008,DKR2012,BERTRAND2016,Lehrenfeld2018,AYLWIN2023,Garcke2025,LI2025}.
A foundational theoretical analysis of isoparametric finite elements on curved domains was
provided by Ciarlet and Raviart \cite{CIARLET1972}, who examined the combined effect of curved boundaries and numerical integration and established asymptotic error estimates. Subsequently, Lenoir \cite{LENOIR1986} pioneered a curved mesh construction by introducing a mapping from the approximate domain $\Omega_h$ to the curved domain $\Omega$, demonstrating its efficacy for the Poisson equation.

In 2008, Cheng and Shu \cite{Cheng2008} established a crucial theoretical and numerical foundation for the necessity of curvilinear elements in achieving genuinely high-order convergence for moving mesh methods. They constructed a third-order conservative Lagrangian scheme on curvilinear meshes for the compressible Euler equations, showing that without geometrically faithful curved elements, the convergence rate saturates at second order even when higher-order numerical reconstructions are employed. Then in 2012, focusing on applications in high energy density physics (radiation transport must be considered), Dobrev et al \cite{DKR2012} proposed a robust and high-order Lagrangian FEM on moving curved meshes, which greatly motivates the efficient solve of RTE on high order meshes. More recently, in \cite{LI2023}, Li et al. made notable contributions to isoparametric methods, including optimal error estimates for arbitrary Lagrangian-Eulerian FEMs on evolving domains using flat interior simplices and piecewise linear mesh velocities. Subsequently, in \cite{Li2024}, a weak discrete maximum principle for elliptic problems on curvilinear polyhedra is analyzed, which leads to a maximum-norm best approximation property for the Poisson equation. For Maxwell's equations on curved domain, in \cite{AYLWIN2023} Alywin et al. examined the influence of numerical integration and employed 3D parametric edge element on curved elements, achieving optimal convergence rates. Kashiwabara developed a finite element analysis for a generalized Robin boundary value problem on curved domains based on the extension approach, proving optimal convergence rates for isoparametric elements without requiring boundary nodes to lie exactly on the true boundary \cite{Kashiwabara2025}.

Compared with symmetric second-order equations, the radiation transport equation in this paper is hyperbolic in nature and is typically posed with inflow boundary conditions. The discontinuous Galerkin (DG) finite element method for this type is widely used and highly efficient for discretizing the radiation transport equation, owing to its inherent ability to handle solution discontinuities and complex geometries \cite{HESTHAVEN2008}. A wealth of studies has been devoted to this topic. Lesaint and Raviart made the fundamental work for neutron transport problems using the DG finite element method \cite{LESAINT1974}. Johnson et al. studied scalar transport equations and established optimal $\mathcal{O}(h^{k+\frac{1}{2}})$ convergence rates in the energy norm \cite{JOHNSON1986,PETERSON1991,BREZZI2004}. A pioneering theoretical contribution came from Houston, Schwab, and S\"{u}li, who developed a comprehensive hp-error analysis for both the streamline-diffusion and the discontinuous Galerkin methods, obtaining estimates that are sharp in both the mesh size $h$ and the polynomial order $k$ \cite{HOUSTON2000}. Then they extended their theoretical framework to advection-diffusion-reaction problems, further advancing the understanding of hp-DG methods in more complex physical regimes \cite{HOUSTON2002}. Guermond and Kanschat further advanced the theoretical understanding by analyzing the asymptotic behavior of upwind DG approximations in the diffusive limit, providing rigorous justification for the method's performance in optically thick regimes \cite{Guermond2010}.

As for the numerical aspects, Woods and Palmer conducted the first systematic investigation of high-order DG methods for radiation transport on meshes with curved surfaces in R-Z geometry, demonstrating that curved meshes do not degrade the expected $k+1$ convergence rates \cite{Woods2019}. Building upon this, Panzer et al. conducted numerical experiments for radiation transport equation on curved domains with reentrant faces \cite{PAZNER2021}. In 2024, Olivier and Haut developed high-order finite element second moment methods on curved meshes, further advancing the numerical capability for complex geometries in transport problems \cite{Olivier2024}. However, these existing work have been largely confined to numerical studies, the rigorous theories analyses and optimal error estimates, particularly with curved inflow boundaries, remain less unexplored. The present work aims to fill this gap by constructing an isoparametric upwind DG scheme on curved meshes with rigorous error analysis.

In this work, we construct an isoparametric mapping $F_h$ \cite{LENOIR1986} from a straight reference domain $\widehat{D}$ to a computational domain $D_h$ that closely conforms to the original curved domain. A key aspect is to control the geometric approximation error so that it does not dominate the total discretization error. Following \cite{LENOIR1986, ARNOLD2020, LI2025}, we introduce an auxiliary mapping $\Phi_h: D \to D_h$ to facilitate the error analysis. The analysis then uses the modified DG norm \cite{HOUSTON2000,PAZNER2021} to quantify how much the numerical solution $I_h$ on $D_h$ deviates from a suitably extended solution $\tilde{I}$ of the original problem \cite{CIARLET1978,BRAMBLE1994}.

We rigorously establish the theoretical framework by decomposing the total error into two parts. The first part is the approximation error of the isoparametric finite element space, for which we define an isoparametric auxiliary operator based on a local projection. The second part is the consistency error of the upwind DG method on curved domains. This consistency error arises from the mismatch between the image of the original inflow boundary $\Phi_h(\Gamma^-)$ under the auxiliary mapping, and the inflow boundary $\Gamma_h^-$ of the computational domain.

The error estimation addresses the following essential challenges from the curved domain geometry and the hyperbolic nature of the RTE, respectively.

\begin{itemize}
    \item Existing error analyses for the RTE on polygonal or polyhedral domains rely on a property that the discrete convection term $\Omega\cdot\nabla I_h$ belongs to the finite element space. While this property does not hold on the isoparametric setting, because functions in the functions in isoparametric finite element space are no longer polynomials.

    \item A fundamental geometric inconsistency arises between the inflow boundary of the computational domain $\Gamma_h^-$ and the image of the original inflow boundary under the auxiliary mapping $\Phi_h(\Gamma^-)$. This discrepancy introduces an additional consistency error that must be carefully controlled to derive the error estimate in the DG norm.

    \item The upwind numerical flux couples with the curved geometry through the unit normal vectors on element interfaces. The isoparametric mapping replaces the true unit normal $\boldsymbol{n}$ with a discrete approximation $\boldsymbol{n}_h$, leading to a geometric consistency error by matching the geometric approximation order with the polynomial order of the finite element space.
\end{itemize}

The remainder of the manuscript is structured as follows. Section 2 introduces the three key domains: the reference, computational, and original domains, along with the auxiliary mapping and the isoparametric finite element space. The standard upwind flux discontinuous Galerkin discretization is then formulated to solve the stationary RTE on the computational domain. In Section 3, we present a theoretical analysis demonstrating that the proposed isoparametric DG method with upwind flux achieves the optimal convergence rate in the DG norm. Section 4 provides numerical experiments that confirm the theoretical results. Concluding remarks are given in Section 5.

\section{Isoparametric finite element method on curved mesh}
\label{sec:2.1}
Throughout this work, we use the classical Sobolev spaces. For a domain $D$, an integer $m \geq 0$, and $1 \leq p \leq \infty$, $W^{m,p}(D)$ denotes the space of functions in $L^p(D)$ whose weak partial derivatives of order up to $m$ also lie in $L^p(D)$. The standard norm and semi-norm on this space are denoted by $||\cdot||_{W^{m,p}(D)}$ and $|\cdot|_{W^{m,p}(D)}$, respectively. The Hilbert space is denoted by $H^m(D) = W^{m,2}(D)$ with norm $||\cdot||_{H^{m}(D)}$ and semi-norm $|\cdot|_{H^{m}(D)}$, respectively. The notation $a\lesssim b$ means that $a\leq Cb$ holds for a constant $C >0$ which is independent of the mesh size $h$.
Moreover, Euclidian norm of a vector $\Omega\in \mathbb{R}^3$ will be denoted by $||\Omega||$.
\subsection{Computational domain and auxiliary mapping}
    We consider a bounded curved domain $D \subset \mathbb{R}^d(d=2,3)$ with a Lipschitz continuous and piecewise $C^{k+1}$ smooth boundary $\Gamma$. To this end, we approximate $D$ by a polyhedral domain $\widehat{D}_h$, which is equipped with a quasi-uniform and shape-regular tetrahedral subdivision $\widehat{\mathcal{T}}_h$ of size at most $h$ and is constructed via a piecewise linear interpolation of $\partial D$ (Sec 4.7 in \cite{BRENNER2008}). Moreover, we denote the boundary of $\widehat{\Gamma}_h = \partial\widehat{D}_h$. The mapping from this polyhedral domain to the original domain is denoted by $F:\widehat{D}_h \rightarrow D$.

    To analyze the extent of approximation between $D$ and $\widehat{D}_h$, we make the following assumptions \cite{LI2025,BERTRAND2014,BERTRAND2016,Lehrenfeld2018}:

    \begin{itemize}
      \item[$(H_1)$] The mapping $F:\widehat{D}_h\rightarrow D$ is piecewise $C^{k+1}$ continuous: $ F|_{\widehat{K}}\in C^{k+1}(\widehat{K}),$ $\forall \widehat{K}\in\widehat{\mathcal{T}}_h$, which satisfies $F(\widehat{K})=\widehat{K}$ if all vertices of $\widehat{K}$ are located within the domain $\widehat{D}_h$.
      \item[$(H_2)$] Let $\mathbb{J}_F$ denote the Jacobian matrix of $F$. Then there holds
          $$||\mathbb{I}-\mathbb{J}_F||_{L^\infty(\widehat{D}_h)}\lesssim h.$$
      \item[$(H_3)$] The mapping $F:\widehat{D}_h\rightarrow D$ is invertible, and let $F^{-1}$ denote the inverse of $F$. Then there holds
    $$\max\limits_{\widehat{K}\in \widehat{\mathcal{T}}_h}(||F||_{W^{k+1,\infty}(\widehat{K})}+||F^{-1}||_{W^{k+1,\infty}(\widehat{K})})\lesssim 1.$$
    \end{itemize}

    The assumptions above are mild and are readily satisfied on condition that the mesh $\widehat{\mathcal{T}}_h$ is sufficiently refined (see, e.g., \cite{BERTRAND2016}). Since $\Gamma$ is piecewise $C^{k+1}$ smooth, assumption $(H_1)$ inherently reflects the smoothness properties of the original boundary $\Gamma$. Assumption $(H_2)$ holds when the approximate boundary $\widehat{\Gamma}_h$ adequately resolves the geometry of the curved boundary $\Gamma$. Indeed, for a sufficiently fine mesh $\mathcal{T}_h$, the Hausdorff distance between $\Gamma$ and $\widehat{\Gamma}_h$ is of order $\mathcal{O}(h^{2})$. We also refer to the remarks following Eq.(5.11) in \cite{BERTRAND2014} for further elaboration. Finally, assumption $(H_3)$ requires that the operator $F$ and its inverse $F^{-1}$ remain uniformly bounded with respect to the mesh size $h$. This condition is reasonable in view of the geometric and functional consistency ensured by Assumptions $(H_1)$ and $(H_2)$.

    For any $\widehat{K}\in\widehat{\mathcal{T}}_h$, let $F_{\widehat{K}}\in\mathbb{P}_k(\widehat{K})$ be the $k^{th}$ order Lagrangian interpolation of $F|_{\widehat{K}}$, where $\mathbb{P}_k$ is the space of vector polynomials with degrees no more than $k$. The global interploation $F_h\in \mathbb{C}(\widehat{D}_h)$ is defined by
    \begin{equation*}
      F_h = F_{\widehat{K}},\quad \forall\widehat{K}\in\widehat{\mathcal{T}}_h.
    \end{equation*}
    It defines the computational domain $D_h$ and therefore a subdivision $\mathcal{T}_h$ (see Fig. \ref{fig:isoparametric_mapping})
    \begin{equation*}
      \mathcal{T}_h=\{K=F_h(\widehat{K}):\widehat{K}\in\widehat{\mathcal{T}_h}\},\qquad
      D_h=interior\left(\bigcup\limits_{K\in\mathcal{T}_h}\overline{K}\right).
    \end{equation*}

    \begin{figure}[htbp]
    \centering
    \label{fig:tet}
    {
        \includegraphics[width=0.35\textwidth]{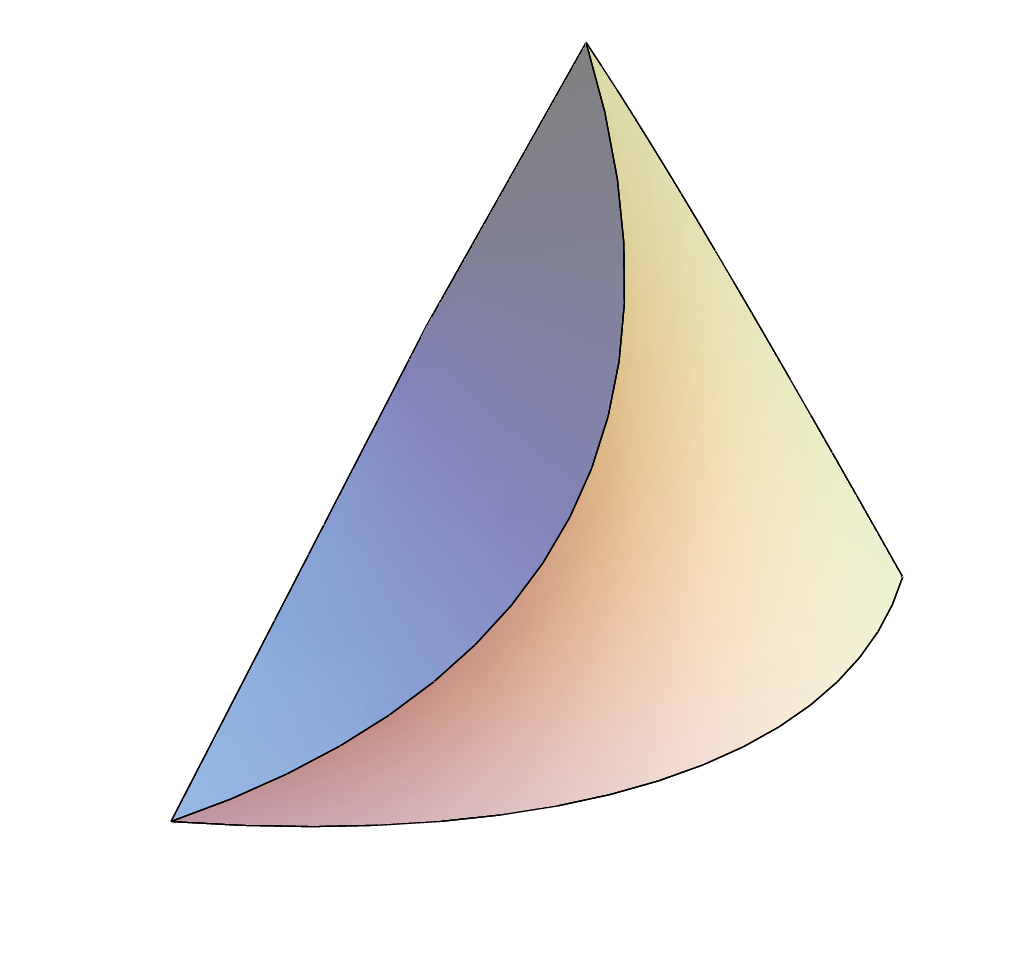}
    }
    \quad
    \label{fig:sphere}
    {
        \includegraphics[width=0.35\textwidth]{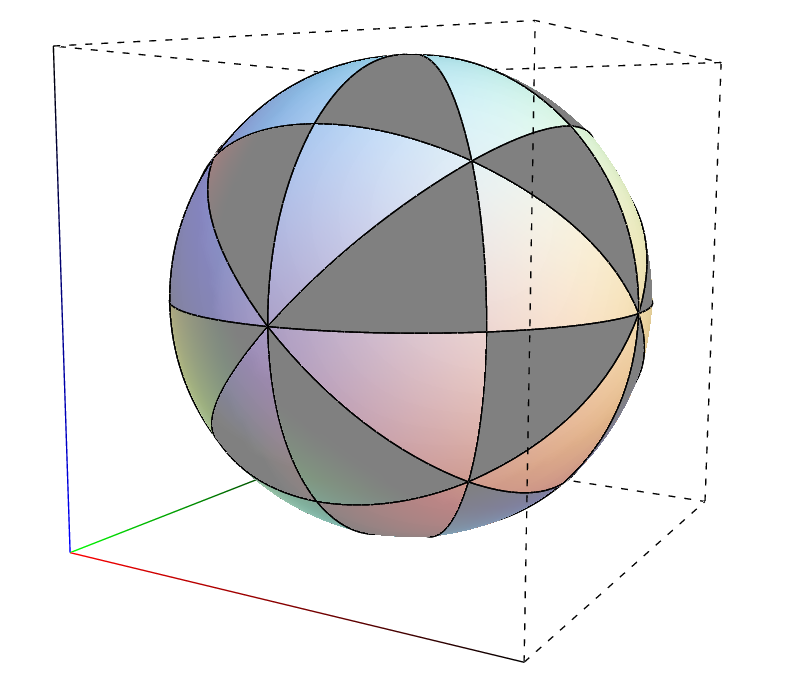}
    }
    \caption{Left: A curved tetrahedron $K$. Right: A sphere domain divided by curved tetrahedra $\mathcal{T}_h$.}
    \label{fig:isoparametric_mapping}
    \end{figure}

    Moreover, we denote the boundary of $D_h$ by $\Gamma_h$ and the inflow boundary by $\Gamma_h^{-}$, namely
    \begin{equation*}
      \Gamma_h = \partial D_h,\qquad\Gamma^{-}_h=\{y\in\Gamma_h:\Omega\cdot \boldsymbol{n}_h(y)<0\}.
    \end{equation*}
    Here $\boldsymbol{n}_h(y)$ is the unit outward normal vector of $\Gamma_h$ at point $y$.
    We also define the inflow (outflow) boundary for each element $K\in \mathcal{T}_h:\partial_{-(+)}K=\{x\in\partial K:\Omega\cdot \boldsymbol{n}_h(x)<(>)0\}$.

    Combining the assumptions $(H_1 \sim H_3)$ with the standard error estimate of Lagrange interpolation \cite{BRENNER2008}, we have
    \begin{equation}\label{eq2.1}
        ||F_h-F||_{W^{m,\infty}(\widehat{D})}\lesssim h^{k+1-m},\quad m=0,1.
    \end{equation}
    Denote the identity mapping as $id$ and the auxiliary mapping $\Phi_h = F_h\circ F^{-1}:D\rightarrow D_h$, respectively, there holds
    \begin{align}\label{eq2.2}
        ||id-\Phi_h||_{W^{m, \infty}(D)} &= ||(F-F_h)\circ F^{-1}||_{W^{m,\infty}(D)}\nonumber\\
        &\leq||F_h-F||_{W^{m,\infty}(\widehat{D})}||\mathbb{J}_F^{-1}||^m\nonumber\\
        &\lesssim h^{k+1-m},\quad m=0,1,
    \end{align}
    which implies the boundedness  of the Jacobian matrix and its determinant estimates
    \begin{equation}\label{eq2.3}
      \|\mathbb{J}_{\Phi_h(x)}\|_{W^{1, \infty}(D)}\lesssim 1,\quad ||\mathcal{J}_{\Phi_h(x)}||_{W^{1, \infty}(D)}\lesssim 1.
    \end{equation}

\subsection{Isoparametric finite element spaces of RTE}
\label{sec:2.2}
    The finite element space on the reference mesh $\widehat{\mathcal{T}}_h$ is defined as
    \begin{equation*}
        \widehat{V}_{h,k} = \bigl\{\widehat{v}_h \in L^2(\widehat{D}_h): \widehat{v}_h|_{\widehat{K}} \in P_k(\widehat{K}),\forall\widehat{K} \in \widehat{\mathcal{T}}_h \bigr\}.
    \end{equation*}
    The corresponding isoparametric finite element space on $\mathcal{T}_h$ is defined as \cite{BRENNER2008}
    \begin{equation*}
        V_{h,k} = \bigl\{ \widehat{v}_h \circ F_h^{-1} : \widehat{v}_h \in \widehat{V}_{h,k} \bigr\}.
    \end{equation*}
    Remarkably, the space $V_{h,k}$ does not have to be a space of polynomials, yet it remains finite-dimensional.

    For \( m \geq 0 \), we introduce the broken Sobolev space
    \begin{equation*}
    H^m(\mathcal{T}_h) = \{ v \in L^2(D_h) : v|_{K} \in H^m(K), \forall K \in \mathcal{T}_h \},
    \end{equation*}
    with the associated broken \( H^m \) norm and the broken \( L^\infty \) norm
    defined respectively as \cite{HESTHAVEN2008}
    \begin{equation*}
    \| v \|_{H^m(\mathcal{T}_h)} = \left( \sum\limits_{K \in \mathcal{T}_h} \| v \|_{H^m(K)}^2 \right)^{\frac{1}{2}}, \quad
    \| v \|_{L^\infty(\mathcal{T}_h)} = \max\limits_{K \in \mathcal{T}_h} \| v \|_{L^\infty(K)}.
    \end{equation*}
    More generally, for any \( p \in [1, \infty] \), the broken \( W^{m,p}(\mathcal{T}_h) \) space and its norms can be defined analogously.

    To characterize the discontinuous nature of the solution across elements, we distinguish between the interior and exterior traces at element boundaries. For each element $K \in \mathcal{T}_h$, a face $F \subset \partial K$, and a function $v \in H^1(\mathcal{T}_h)$, the interior and exterior traces on $F$ are denoted by $v^{+}$ and $v^{-}$, respectively. Furthermore, the jump of $v$ across $F$ is defined as $\llbracket v \rrbracket := v^{+} - v^{-}$ if $F$ is an inner face.

    Now we propose the isoparametric upwind DG formulation for (\ref{eq1.1}),
    which is defined as follows:
    Find $I_h\in V_{h,k}$ such that
    \begin{equation}\label{eq2.4}
    \mathcal{B}(I_h,v_h) = l(v_h),\quad \forall v_h\in V_{h,k},
    \end{equation}
    where the bilinear form $\mathcal{B}(\cdot, \cdot)$ with upwind flux and the linear functional $l(\cdot)$ are defined respectively as
    \begin{align}
    \mathcal{B}(I_h,v_h)=&
    \sum_{K\in\mathcal{T}_h}
    \left(
    \int_{K}(\Omega\cdot\nabla I_h+\sigma I_h)v_hdx \right)\nonumber
    -\sum_{K\in\mathcal{T}_h}\left(\int_{\partial_{-}K\cap\Gamma^{-}_h}(\Omega\cdot \boldsymbol{n_h})I_h^{+}v_h^{+}ds \right)\\
    &-\sum_{K\in\mathcal{T}_h}\left(\int_{\partial_{-}K \setminus \Gamma_h}(\Omega\cdot \boldsymbol{n_h}) \llbracket I_h \rrbracket v_h^{+}ds
    \right)\label{eq2.5},\\
    l(v_h) =& \sum_{K\in\mathcal{T}_h}
    \left(
    \int_{K}f_hv_hdx\right) -
    \sum_{K\in\mathcal{T}_h}
    \left(\int_{\partial_{-}K\cap\Gamma^{-}_h}(\Omega\cdot \boldsymbol{n_h})g_hv_h^{+}ds
    \right)\label{eq2.6}.
    \end{align}
    Here $f_h=f\circ\Phi_h^{-1}$, $g_h=\tilde{g}\circ\Phi_h^{-1}$ and $\tilde{g}$ is the zero extension of $g$ on $\Gamma$. We remark that $\mathcal{B}(\cdot, \cdot)$ is the extension of the bilinear form defined for convex polyhedra in
    the seminal paper \cite{HOUSTON2002}, now allowing $F$ to be a curved face.

    Employing  integration by parts to (\ref{eq2.5}) yields an alternative expression for the bilinear form
    \begin{align}    \label{eq2.7}
    \mathcal{B}(I_h,v_h) = &\sum_{K\in\mathcal{T}_h}
    \left(
    \int_{K}(\sigma v_h-\Omega\cdot\nabla v_h)I_hdx \right)
    +\sum_{K\in\mathcal{T}_h}
    \left(\int_{\partial_{+}K\cap\Gamma_h}(\Omega\cdot \boldsymbol{n_h})I_h^{+}v_h^{+}ds\right) \nonumber\\
    &+ \sum_{K\in\mathcal{T}_h}
    \left(\int_{\partial_{-}K \setminus \Gamma_h}(\Omega\cdot \boldsymbol{n_h}) \llbracket v_h \rrbracket I_h^{-}ds
    \right).
    \end{align}

    The DG norm $||\cdot||_{DG}$ in this paper is defined as follows
    \begin{align}
    ||v_h||_{DG}^2 = \mathcal{B}(v_h,v_h)
    = &\sum_{K\in\mathcal{T}_h}
    \left(
    \int_{K}\sigma v_h^2dx\right)
    + \frac{1}{2}\sum_{K\in\mathcal{T}_h}
    \left(\int_{\Gamma_h\cap\partial K}|\Omega\cdot \boldsymbol{n_h}| v_h^2ds \right)\nonumber\\
    &- \frac{1}{2}\sum_{K\in\mathcal{T}_h}
    \left(\int_{\partial_{-}K \setminus \Gamma_h}\Omega\cdot \boldsymbol{n_h}\llbracket v_h\rrbracket^2ds
    \right).\nonumber
    \end{align}

    \begin{remark}
        The upwind DG numerical scheme in \cite{HOUSTON2000} for the RTE on polygonal or polyhedral domains is consistent. The classical approach therefore estimates
        \begin{align*}
          \|I - I_h\|_{DG} \leq
          \|I - \widehat{P}I\|_{DG} + \|\widehat{P}I - I_h\|_{DG},
        \end{align*}
        where $\widehat{P}:L^2(\widehat{\mathcal{T}}_h)\rightarrow \widehat{V}_{h,k}$ is the local $L^2$ projection operator and the second term is controlled by $\|I - \widehat{P}I\|_{DG}$ \cite{JOHNSON1986, BREZZI2004, PAZNER2021}.
        However, for the numerical scheme in (\ref{eq2.6}), we introduce the concepts of extension and auxiliary mapping, which makes the scheme inconsistent and thus introduces extra difficulties for error estiamte.
    \end{remark}

    Inspired by the stability analysis in Lemma 2.2 of \cite{HOUSTON2000} and Lemma 3.1 of \cite{HOUSTON2002}, we extend it to the isoparametric case.
    \begin{lemma}\label{lemma2.1}
    Assume that there exist two positive constants $m$ and $M$ such that $0<m\leq\sigma\leq M$. Then the numerical solution $I_h$ of (\ref{eq2.4}) obeys the bound
    \begin{align}
    \sum_{K\in\mathcal{T}_h}
    &\left(
    \int_{K}\sigma I_h^2dx
    +\int_{\partial_{-}K\cap\Gamma^{-}_h}|\Omega\cdot \boldsymbol{n_h}| I_h^2ds
    +\int_{\partial_{-}K \setminus \Gamma_h}|\Omega\cdot \boldsymbol{n_h}| \llbracket I_h\rrbracket^2ds\right)\nonumber\\
    +&\sum_{K\in\mathcal{T}_h}\left(\int_{\partial_{+}K\cap\Gamma_h}\Omega\cdot \boldsymbol{n_h} I_h^2ds
    \right)
    \lesssim\sum_{K\in\mathcal{T}_h}\left(
    \int_{K}f_h^2dx+\int_{\partial_{-}K\cap\Gamma^{-}}g_h^2ds
    \right).\label{eq2.8}
    \end{align}
    \end{lemma}
    \begin{remark}
      The Lemma \ref{lemma2.1} implies the uniqueness of the solution of the isoparametric upwind DG method (\ref{eq2.4}). Moreover, the existence of the solution $I_h$ comes from its uniqueness because (\ref{eq2.4}) is a linear problem over the finite-dimensional space $V_{h,k}$.
    \end{remark}

\section{Error analysis of the isoparametric finite element method}
\label{sec:3}
    Now we will present our main contribution, namely the rigorous error estimate of the isoparametric
    upwind DG method \eqref{eq2.4}. This section will analyze the error under the DG norm between the extension of the true solution $I$ in (\ref{eq1.1}) and $I_h$ of (\ref{eq2.4}) on $\mathcal{T}_h$.

    Inspired by the work in \cite{AYLWIN2023, CIARLET1978}, we extend the true solution $I$ of \eqref{eq1.1} to a function $\tilde I$ on the hold-all domain $D_H = \operatorname{convex}(D \cup D_h)$ such that $\tilde I \in H^{k+1}(\mathcal{T}_h)$ and $\tilde I = I$ on $D$. For $k\ge2$ and $d=2,3$ we have $k+1>d/2$, thus for each element $K\in\mathcal{T}_h$, the Sobolev embedding theorem implies $\tilde I|_K \in C^0(\overline{K})$ \cite{Adams2003}. Hence $\tilde I$ is bounded on every $K$ and consequently $\|\tilde I\|_{L^\infty(\Gamma_h)}\le C$. For convenience we also set $\tilde f = \Omega\cdot\nabla\tilde I + \sigma\tilde I$.
\subsection{Approximation error in the isoparametric finite element space}
\label{Section3.1}

    Given the finite element space $\widehat{V}_{h,k}$ and the local $L^2$ projection operator $\widehat{P}:L^2(\widehat{\mathcal{T}}_h)\rightarrow \widehat{V}_{h,k}$ \cite{BRENNER2008}, we define a corresponding isoparametric auxiliary operator $\Lambda:H^{k+1}(\mathcal{T}_h)\rightarrow V_{h,k}$. It is defined by
    \begin{equation}\label{eq3.0}
        \Lambda v :=(\widehat{P}\hat{v})\circ F_h^{-1}
        =[\widehat{P}(v\circ F_h)]\circ F_h^{-1}.
    \end{equation}
    This definition follows the standard isoparametric construction: the pull-back $\hat{v} = v \circ F_h$ is firstly projected onto the straight element $\widehat{K}$ by $\widehat{P}$, and the result is then pushed forward via $F_h^{-1}$ to the curved element $K_h$, yielding a function in the isoparametric finite element space $V_{h,k}$.

    \begin{lemma}\label{lemma3.1}
    For $\tilde{I}\in H^{k+1}(\mathcal{T}_h)$ and $\Lambda: L^{2}(\mathcal{T}_h)\rightarrow V_{h,k}$ defined above,
    there holds
    \begin{equation}
        ||\tilde{I} - \Lambda \tilde{I}||_{DG}\lesssim h^{k+1/2}||\tilde{I}||_{H^{k+1}(\mathcal{T}_h)}.\label{eq3.1}
    \end{equation}
    \end{lemma}
    \begin{proof}\quad
        Based on the definition of $\Lambda$, together with the uniform boundedness of the Jacobian determinant $\mathcal{J}_{F_h}$ and its reciprocal (cf. Assumption $H_2$), we will obtain
        \begin{align}\label{eq3.2}
            ||\tilde{I} - \Lambda\tilde{I}||_{L^2(\mathcal{T}_h)}
            &=||\tilde{I} - [\widehat{P}(\tilde{I}\circ F_h)]\circ F_h^{-1}||_{L^2(\mathcal{T}_h)}\nonumber\\
            &=||\mathcal{J}_{F_h}^{1/2}(\tilde{I}\circ F_h - \widehat{P}(\tilde{I}\circ F_h))||_{L^2(\mathcal{\widehat{T}}_h)}\nonumber\\
            &=||\mathcal{J}_{F_h}^{1/2}(\hat{\tilde{I}} - \widehat{P}\hat{\tilde{I}})||_{L^2(\mathcal{\widehat{T}}_h)}
            \lesssim ||\hat{\tilde{I}} - \widehat{P}\hat{\tilde{I}}||_{L^2(\mathcal{\widehat{T}}_h)}\nonumber\\
            &\lesssim h^{k+1}||\hat{\tilde{I}}||_{H^{k+1}(\widehat{\mathcal{T}}_h)}
            \lesssim h^{k+1}||\tilde{I}||_{H^{k+1}(\mathcal{T}_h)}.
        \end{align}
        The penultimate inequality is a consequence of the approximation property of the local $L^2$ projection $\widehat{P}$ on $\widehat{V}_{h,k}$ \cite{BRENNER2008}, whereas the last inequality results from Eq. (\ref{eq2.2}) and the assumption $H_3$.

        Let $\mathcal{E}_h = \big\{\bigcup{\partial K:K\in\mathcal{T}_h}\big\}$ and $\widehat{\mathcal{E}}_h = \big\{\bigcup{\partial \widehat{K}:\widehat{K}\in\widehat{\mathcal{T}}_h}\big\}$ denote the unions of all element boundaries on $\mathcal{T}_h$ and $\widehat{\mathcal{T}}_h$, respectively. Similar to (\ref{eq3.2}) and invoking the trace theorem \cite{BRENNER2008}, we arrive at
        \begin{align}\label{eq3.3}
            &|||\Omega\cdot \boldsymbol{n}_h|^{1/2}\llbracket\tilde{I} - \Lambda\tilde{I}\rrbracket||_{L^2(\mathcal{E}_h)}\nonumber\\
            &=|||\Omega\cdot \boldsymbol{n}_h|^{1/2}\llbracket \tilde{I} - (\widehat{P}\hat{\tilde{I}})\circ F_h^{-1} \rrbracket||_{L^2(\mathcal{E}_h)}\nonumber\\
            &=|||\Omega\cdot\mathbb{B}_h \hat{\boldsymbol{n}}|^{1/2}\llbracket \hat{\tilde{I}} - \widehat{P}\hat{\tilde{I}} \rrbracket||_{L^2(\widehat{\mathcal{E}}_h)}\nonumber\\
            &\lesssim ||\llbracket \hat{\tilde{I}} - \widehat{P}\hat{\tilde{I}} \rrbracket||_{L^2(\widehat{\mathcal{E}}_h)}
            \lesssim h^{-1/2} ||\hat{\tilde{I}} - \widehat{P}\hat{\tilde{I}}||_{H^{k+1}(\widehat{\mathcal{T}}_h)}\nonumber\\
            &\lesssim h^{k+1/2}||\hat{\tilde{I}}||_{H^{k+1}(\widehat{\mathcal{T}}_h)}
            \lesssim h^{k+1/2}||\tilde{I}||_{H^{k+1}(\mathcal{T}_h)},
        \end{align}
        where $\mathbb{B}_h=||\mathbb{J}_{F_h}^{T}\boldsymbol{n}_h||\mathcal{J}_{F_h}\mathbb{J}_{F_h}^{-T}$ denotes the scaled Jacobian matrix, and $\hat{\boldsymbol{n}} = \mathbb{J}_{F_h}^{T}\boldsymbol{n}_h / ||\mathbb{J}_{F_h}^{T}\boldsymbol{n}_h||$ is the unit normal vector on $\widehat{\mathcal{T}}_h$ \cite{LI2025}.

        Combining (\ref{eq3.2}), (\ref{eq3.3}) and the definition of the DG norm yields the estimates (\ref{eq3.1}). The proof of Lemma \ref{lemma3.1} is completed.
    \end{proof}
    To control the projection error $\|\widehat{P}I - I_h\|_{DG}$, the analysis for the RTE on polygonal or polyhedral domains relies on the property that DG continuity of the bilinear form for the projection error, which necessitates the discrete convection term $\Omega\cdot\nabla \hat{v}_h$ belongs to $\widehat{V}_{h,k}$.
    However, for isoparametric elements we cannot guarantee that $\Omega\cdot\nabla v_h$ belongs to $\in V_{h,k}$.
    Additional techniques are required to address this issue.

    Next we prove a key lemma for our error estimate, which states that the bilinear form is continuous with respect
    to the isoparametric projection error.
    \begin{lemma}\label{lemma3.2}
    Let $\tilde{I}$ be the extension solution of (\ref{eq1.1}) and let $\Lambda:L^2(\mathcal{T}_h)\rightarrow V_{h,k}$ denote the isoparametric auxiliary operator. For any $v_h \in V_{h,k}$, with the bilinear form $\mathcal{B}(\cdot,\cdot)$ defined in (\ref{eq2.5}), there holds

    \begin{equation}\label{eq3.4}
    \mathcal{B}(\tilde{I} - \Lambda\tilde{I}, v_h) \lesssim
    \|\tilde{I} - \Lambda\tilde{I}\|_{\mathrm{DG}} \,
    \|v_h\|_{\mathrm{DG}} .
    \end{equation}
    \end{lemma}

    \begin{proof}\quad
    Utilizing the equivalent form of $\mathcal{B}(\cdot, \cdot)$ in (\ref{eq2.7}), we can decompose it into four parts
    \begin{align}\label{eq3.5}
    \mathcal{B}(\tilde{I} - \Lambda\tilde{I},v_h)
    =& \sum_{K\in\mathcal{T}_h}
    \int_{K}\sigma\, v_h(\tilde{I} - \Lambda\tilde{I})dx
    - \sum_{K\in\mathcal{T}_h}
    \int_{K}\Omega\cdot\nabla v_h(\tilde{I} - \Lambda\tilde{I})dx \nonumber\\
    &+ \sum_{K\in\mathcal{T}_h}\int_{\partial_{+}K\cap\Gamma_h}
    (\Omega\cdot \boldsymbol{n_h})(\tilde{I} - \Lambda\tilde{I})^{+}v_h^{+}ds \nonumber\\
    &+ \sum_{K\in\mathcal{T}_h}\int_{\partial_{-}K \setminus \Gamma_h}
    (\Omega\cdot \boldsymbol{n_h})\llbracket v_h \rrbracket (\tilde{I} - \Lambda\tilde{I})^{-}\,ds \nonumber\\
    =&: L_1 + L_2 + L_3 + L_4 .
    \end{align}

    From the definition of the DG norm, together with the Cauchy-Schwarz inequality and the boundedness of $\sigma$, it follows directly that
    \begin{equation*}
    L_1 + L_3 + L_4 \lesssim
    \|\tilde{I} - \Lambda\tilde{I}\|_{\mathrm{DG}}
    \|v_h\|_{\mathrm{DG}} .
    \end{equation*}

    Note that due to the non-affine transformation,
    the term $\int_{K}\Omega\cdot\nabla v_h(\tilde{I} - \Lambda\tilde{I})dx$ in $L_2$
    will generally not vanish.
    However for straight meshes and affine DG method, this is not the case.
    For $L_2$, applying the integral transformation $F_h$ yields
    \begin{align}\label{eq3.6}
        L_2 = & -\sum_{K\in\mathcal{T}_h}\int_{K}
        \bigl(\Omega\cdot\nabla v_h\bigr)
        \bigl[\tilde{I} - (\widehat{P}\hat{\tilde{I}})\circ F_h^{-1}\bigr]dx \nonumber\\
        = &-\sum_{\widehat{K}\in\widehat{\mathcal{T}}_h}\int_{\widehat{K}}
        \bigl(\Omega\cdot\mathbb{J}_{F_h}^{-T}\nabla \hat{v}_h\bigr)
        (\hat{\tilde{I}} - \widehat{P} \hat{\tilde{I}})\,\mathcal{J}_{F_h}d\hat{x} \nonumber\\
        = &-\sum_{\widehat{K}\in\widehat{\mathcal{T}}_h}\int_{\widehat{K}}
        \Bigl[\Omega\cdot\bigl(\mathbb{J}_{F_h}^{-T}-\mathbb{I}\bigr)\nabla \hat{v}_h\Bigr]
        (\hat{\tilde{I}} - \widehat{P} \hat{\tilde{I}})\,\mathcal{J}_{F_h}d\hat{x}\nonumber\\
        & +\sum_{\widehat{K}\in\widehat{\mathcal{T}}_h}\int_{\widehat{K}}
        \bigl(\Omega\cdot\nabla \hat{v}_h\bigr)
        (\hat{\tilde{I}} - \widehat{P} \hat{\tilde{I}})\,(1-\mathcal{J}_{F_h})\,d\hat{x} \nonumber\\
        & -\sum_{\widehat{K}\in\widehat{\mathcal{T}}_h}\int_{\widehat{K}}
        \bigl(\Omega\cdot\nabla \hat{v}_h\bigr)
        (\hat{\tilde{I}} - \widehat{P} \hat{\tilde{I}})\,d\hat{x} =: R_1 + R_2 + R_3 .
    \end{align}

    From assumptions $H_2$ and $H_3$, we have
    \begin{equation*}
    \|\mathbb{J}_{F_h}^{-T}-\mathbb{I}\|_{L^\infty(\widehat{D}_h)} \lesssim h,\quad
    \|1-\mathcal{J}_{F_h}\|_{L^\infty(\widehat{D}_h)} \lesssim h,\quad
    \|\mathcal{J}_{F_h}\|_{L^\infty(\widehat{D}_h)} \lesssim 1.
    \end{equation*}
    Furthermore, employing the inverse inequality on $\widehat{V}_{h,k}$ \cite{BRENNER2008} to eliminate the gradient term $\nabla \hat{v}_h$, we deduce
    \begin{align}\label{eq3.7}
        R_1 + R_2 &\lesssim h\,
        \|\nabla \hat{v}_h\|_{L^2(\widehat{\mathcal{T}}_h)}\,
        \|\hat{\tilde{I}} - \widehat{P} \hat{\tilde{I}}\|_{L^2(\widehat{\mathcal{T}}_h)}
        \lesssim \|\hat{v}_h\|_{L^2(\widehat{\mathcal{T}}_h)}\,
        \|\hat{\tilde{I}} - \widehat{P} \hat{\tilde{I}}\|_{L^2(\widehat{\mathcal{T}}_h)} \nonumber\\
        &\lesssim \|v_h\|_{L^2(\mathcal{T}_h)}\,
        \|\tilde{I} - \Lambda\tilde{I}\|_{L^2(\mathcal{T}_h)} .
    \end{align}

    Finally, note that
    \[\Omega\cdot\nabla \hat{v}_h \in \widehat{V}_{h,k}.\]
    By the definition of the local $L^2$ projection onto $\widehat{V}_{h,k}$, we have $R_3 = 0$.
    Combining (\ref{eq3.5}), (\ref{eq3.6}) and (\ref{eq3.7}) yields the desired estimate, which completes the proof.
    \end{proof}

\subsection{Analysis of the consistency error for the upwind DG method}
\label{Section3.2}
    This subsection is devoted to estimating the consistency error $||\Lambda I - I_h||_{DG}$, which is closely linked to the discrete formulation (\ref{eq2.4}). To properly deal with $f_h$ and $g_h$ in the linear form $l(\cdot)$ in (\ref{eq2.6}), the errors arising from solution extension and geometric approximation must be considered separately.

    To characterize the effect of the auxiliary mapping $\Phi_h$ on the inflow boundary, we define the transformed solution $I^\sharp = I \circ \Phi_h^{-1}$, which belongs to $H^1(D_H)$ \cite{BRAMBLE1994}. It satisfies $I^\sharp = g_h$ on the mapped inflow boundary $\Phi_h(\Gamma^{-})$.
\subsubsection{Estimation of the difference between $\tilde{I}$ and $I^\sharp$}
\label{Section3.2.1}
    \begin{lemma}\label{lemma3.3}
    Let $\tilde{I}$ be the extension of exact solution $I$, and let $I^\sharp = I \circ \Phi_h^{-1}\in H^1(\mathcal{T}_h)$ denote its transformed counterpart on the computational domain $D_h$. It holds that
    \begin{equation}\label{eq3.8}
    \|\tilde{I} - I^\sharp\|_{DG}
    \lesssim h^{k+\frac{1}{2}} \|\tilde{I}\|_{L^2(\mathcal{T}_h)} + h^{k+\frac{3}{2}} \|\tilde{I}\|_{H^2(\mathcal{T}_h)}.
    \end{equation}
    \end{lemma}

    \begin{proof}\quad
    Employing the definition of $I^\sharp$ and (\ref{eq2.2}) yields
    \begin{align}\label{eq3.9}
    \|\tilde{I} - I^\sharp\|_{L^2(\mathcal{T}_h)}^2
    &= \|\tilde{I} - I \circ \Phi_h^{-1}\|_{L^2(\mathcal{T}_h)}^2
    = \|\tilde{I} - \tilde{I} \circ \Phi_h^{-1}\|_{L^2(\mathcal{T}_h)}^2 \nonumber \\
    &\leq \| \mathrm{id} - \Phi_h^{-1} \|_{L^\infty(D_h)}^2 \|\tilde{I}\|_{L^2(\mathcal{T}_h)}^2
    \leq h^{2k+2} \|\tilde{I}\|_{L^2(\mathcal{T}_h)}^2.
    \end{align}

    Differentiating $I^\sharp$ and applying the chain rule provides its gradient representation
    \begin{equation*}
    \nabla I^\sharp = \nabla (I \circ \Phi_h^{-1})
    = \mathbb{J}_{\Phi_h}^{-T} (\nabla I) \circ \Phi_h^{-1}
    = \mathbb{J}_{\Phi_h}^{-T} (\nabla \tilde{I}) \circ \Phi_h^{-1} , \quad in \ \  D_h.
    \end{equation*}
    Then it yields the gradient error estimate
    \begin{align} \label{eq3.10}
    \|\nabla \tilde{I} - \nabla I^\sharp\|_{L^2(\mathcal{T}_h)}^2
    &= \|\nabla \tilde{I} - \mathbb{J}_{\Phi_h}^{-T}(\nabla \tilde{I}) \circ \Phi_h^{-1}\|_{L^2(\mathcal{T}_h)}^2 \nonumber \\
    &\leq \| (\mathbb{I} - \mathbb{J}_{\Phi_h}^{-T})\nabla \tilde{I}\|_{L^2(\mathcal{T}_h)}^2
    + \|\mathbb{J}_{\Phi_h}^{-T}(\nabla \tilde{I} - \nabla \tilde{I} \circ \Phi_h^{-1})\|_{L^2(\mathcal{T}_h)}^2 \nonumber \\
    &\lesssim h^{2k} \|\tilde{I}\|_{H^1(\mathcal{T}_h)}^2 + h^{2k+2} \|\tilde{I}\|_{H^{2}(\mathcal{T}_h)}^2,
    \end{align}
    where the first term reflects the auxiliary mapping approximation from (\ref{eq2.2}) as well as the second term stems from the translation of the gradient field.
    Therefore, by the trace theorem in \cite{BRENNER2008} and the estimates (\ref{eq3.9}) and (\ref{eq3.10}), the DG norm error satisfies
    \begin{align} \label{eq3.11}
    \|\tilde{I} - I^\sharp\|_{DG}^2
    &\lesssim \|\tilde{I} - I^\sharp\|_{L^2(\mathcal{T}_h)}^2 + \|\tilde{I} - I^\sharp\|_{L^2(\mathcal{E}_h)}^2 \nonumber \\
    &\lesssim \|\tilde{I} - I^\sharp\|_{L^2(\mathcal{T}_h)}^2
    + h^{-1}\|\tilde{I} - I^\sharp\|_{L^2(\mathcal{T}_h)}^2 + h \|\tilde{I} - I^\sharp\|_{H^1(\mathcal{T}_h)}^2 \nonumber \\
    &\lesssim h^{2k+1} \|\tilde{I}\|_{L^2(\mathcal{T}_h)}^2 + h^{2k+3} \|\tilde{I}\|_{H^2(\mathcal{T}_h)}^2.
    \end{align}
    The proof of Lemma \ref{lemma3.3} is completed.
    \end{proof}

\subsubsection{Geometric approximation error of inflow boundary}
\label{Section3.2.2}
    The estimate of the geometric approximation error between the inflow boundary $\Gamma^{-}_h$ of $D_h$ and $\Phi_h(\Gamma^{-})$ relies on the following mild assumptions \cite{Lee2013}

    \begin{itemize}
      \item[$(A_1)$] $\Gamma_h$ is a compact, oriented, piecewise $C^{k+1}$ smooth regular surface.
      \item[$(A_2)$] 0 is the regular value of the function $f(x) = \Omega\cdot \boldsymbol{n}_h(x)$, that is, $\forall x\in f^{-1}(0)$, $\nabla f(x)\neq 0$.
    \end{itemize}

    The main difficulty is that $\Gamma_h^-\setminus\Phi_h(\Gamma^-)$ is a subset of $\Gamma_h$, yet the information we could use is the pointwise condition between $\boldsymbol{n}_h$ and $\boldsymbol{n}$, which in turn is governed by the Jacobian matrix $\mathbb{J}_{\Phi_h}$ and its deviation from the identity $\mathbb{I}$. Moving from a pointwise condition on the normal vector to a quantitative bound on the measure of a subset is a nontrivial step that requires additional care. To address this issue, we introduce the set
    \[
    T_{h^k} = \{ x \in \Gamma_h : |\Omega\cdot \boldsymbol{n}_h(x)| \lesssim h^k \}.
    \]
    The analysis then proceeds in two steps. First, we show that the measure of $T_{h^k}$ is bounded by order $h^k$. Second, we prove that the portion of $\Gamma^-_h \backslash \Phi_h(\Gamma^-)$ can be controlled by $T_{h^k}$. The desired estimate follows.

    \begin{lemma}\label{lemma3.4}
        Suppose that the assumptions $A_1$ and $A_2$ hold, and let $h > 0$ be sufficiently small. Then the area of the set
        $T_{h^k} = \{ x \in \Gamma_h : |\Omega\cdot \boldsymbol{n}_h(x)| \lesssim h^k \}$ satisfies
        \begin{align}
        \mathrm{Area}(T_{h^k}) \lesssim h^k.
        \label{eq3.12}
        \end{align}
    \end{lemma}

    \begin{proof}\quad
        Let $\mathcal{E}^c_h$ denote the set of curved boundary faces of $\Gamma_h$. For any $F\in\mathcal{E}^c_h$, assumption $(A_2)$ implies that $0$ is a regular value of $\Omega\cdot\boldsymbol{n}_h|_F$. Consequently, the zero level set $\Gamma^0_F := f|_F^{-1}(0)$, if non-empty, is a finite union of $C^1$ $(n-2)$-dimensional submanifolds of $F$.

        We now focus on faces for which $\Gamma^0_F$ is non-empty. By the regular value condition $\nabla f \neq 0$ on $\Gamma^0_F$, there exists a tubular neighborhood of $\Gamma^0_F$ within $F$ (see, e.g., \cite{Lee2013}). In this neighborhood we can introduce coordinates $(s, r)$, where $s$ parametrizes $\Gamma^0_F$ and $r$ is a normal coordinate (in the direction of $\nabla f$). These coordinates satisfy the two-sided estimate
        \begin{align*}
        c_1 |r| \leq |f(x)| \leq c_2 |r|, \quad \text{and} \quad |r| \approx \operatorname{dist}(x, \Gamma^0_F),
        \end{align*}
        with positive constants $c_1, c_2$ independent of $h$ and the face $F$. Consequently, for any $x \in F$ with $|f(x)| \lesssim h^k$, we have $|r(x)| \lesssim h^k$. Recalling the definition $T_{h^k} = { x\in\Gamma_h : |f(x)| \lesssim h^k }$, it follows that
        \begin{align*}
        F_*:=T_{h^k} \cap F \subset \{ x : |r(x)| \lesssim h^k \}.
        \end{align*}

        In the $(s,r)$-coordinates, the area element on the face $F$ is given by $dS = J(s,r)dsdr$, where $J(s,0)=1$ and $|J(s,r)-1| \lesssim |r|$ due to the smoothness of $F$. Integrating over the tubular neighborhood $F_* = T_{h^k} \cap F$ yields the local estimate
        \begin{align}\label{eq3.13.5}
        \mathrm{Area}(T_{h^k} \cap F) \lesssim \mathrm{Length}(\Gamma^0_F \cap F) \cdot h^k.
        \end{align}

        The global estimate follows by summing over all curved boundary faces $\mathcal{E}^c_h$. By the shape-regularity and quasi-uniformity of $\mathcal{T}_h$, the total length of the discrete zero level sets remains bounded independently of $h$
        \begin{align*}
          \sum\limits_{F\in\mathcal{F}_h}Length(\Gamma_F^0)\lesssim 1.
        \end{align*}

        Summing the local estimates \eqref{eq3.13.5} over all faces, we obtain
        \begin{align}
        \mathrm{Area}(T_{h^k}) = \sum_{F \in \mathcal{F}_h} \mathrm{Area}(T_{h^k} \cap F) \lesssim \left( \sum_{F \in \mathcal{F}_h} \mathrm{Length}(\Gamma^0_F \cap F) \right) h^k \lesssim h^k.\label{eq3.13}
        \end{align}
        This completes the proof of Lemma \ref{lemma3.4}.
    \end{proof}
    The following lemma is essential for characterizing the boundedness of $\Gamma^{-}_h \backslash \Phi_h(\Gamma^{-})$.

    \begin{lemma}\label{lemma3.5}
        Let $A \Delta B = (A\setminus B)\cup (B\setminus A)$ denote the symmetric difference of two sets $A$ and $B$. Suppose $\Gamma^{-}$ and $\Gamma^{-}_h$ are the inflow boundaries of $D$ and $D_h$, respectively. Then the area of symmetric difference under the mapping $\Phi_h$ satisfies
        \begin{equation}\label{eq3.14.1}
            Area\left( \Gamma^{-}_h  \Delta  \Phi_h(\Gamma^{-}) \right) \lesssim h^k.
        \end{equation}
    \end{lemma}
    \begin{proof}\quad
        Inspired by Eq.(A.20) in \cite{Lehrenfeld2018}, the relation between the outward normal vectors $\boldsymbol{n}$ on $\Gamma$ and $\boldsymbol{n}_h$ on $\Gamma_h$ is given by $\boldsymbol{n}_h(x) = \mathbb{J}_{\Phi_h}^{-T} \boldsymbol{n} / \|\mathbb{J}_{\Phi_h}^{-T} \boldsymbol{n}\|$. Let $\boldsymbol{n}_s = \boldsymbol{n} / \|\mathbb{J}_{\Phi_h}^{-T} \boldsymbol{n}\|$ denote the scaled normal on $\Gamma$. Motivated by Eq.(2.17) in \cite{Kashiwabara2025}, we apply the mapping estimate (\ref{eq2.2}) and Proposition 2 in \cite{LENOIR1986} to arrive at
        \begin{align}\label{eq3.14.2}
            \|\boldsymbol{n}_s \circ \Phi_h^{-1} - \boldsymbol{n}_{h}\|_{L^{\infty}(\Gamma_h)}
            &= \|\mathbb{J}^{T}_{\Phi_h}\boldsymbol{n}_h - \boldsymbol{n}_h\|_{L^{\infty}(\Gamma_h)} \nonumber\\
            &= \|(\mathbb{J}^{T}_{\Phi_h} - \mathbb{I})\boldsymbol{n}_h\|_{L^{\infty}(\Gamma_h)}
            \lesssim h^k.
        \end{align}
        Consequently, we have
        \begin{align}\label{eq3.14}
            |\Omega \cdot \boldsymbol{n}_s \circ\Phi_h^{-1}(x) - \Omega \cdot \boldsymbol{n}_{h}(x)| \lesssim h^k , \quad \forall x \in \Gamma_h.
        \end{align}

        Recall the definitions of the inflow boundaries
        \[
        \Gamma^{-} = \{x \in \Gamma : \Omega \cdot \boldsymbol{n}(x) < 0\},\qquad
        \Gamma^{-}_h = \{x \in \Gamma_h : \Omega \cdot \boldsymbol{n}_h(x) < 0\}.
        \]
        Take a point $x \in \Gamma_h^{-} \setminus \Phi_h(\Gamma^{-})$. Since $x \notin \Phi_h(\Gamma^{-})$, we have $\Phi_h^{-1}(x) \notin \Gamma^{-}$, which implies $\Omega \cdot \boldsymbol{n}_s \circ \Phi_h^{-1}(x) \ge 0$. A symmetric argument for $x \in \Phi_h(\Gamma^{-}) \setminus \Gamma_h^{-}$ yields the analogous inequality. Consequently, the symmetric difference $\Gamma_h^{-} \Delta \Phi_h(\Gamma^{-})$ is contained in the $h^k$-tube
        \[
        T_{h^k} = \{ x\in \Gamma_h : |\Omega\cdot\boldsymbol{n}_h(x)| \lesssim h^k \}.
        \]
        Applying Lemma \ref{lemma3.4} to this inclusion yields $Area(\Gamma_h^{-} \Delta \Phi_h(\Gamma^{-})) \lesssim h^k$, which completes the proof of (\ref{eq3.14.1}).
    \end{proof}

    \begin{lemma}\label{lemma3.6}
        Let $\mathcal{E}^0_h$ denote the set of boundary faces $F$ such that $\overline{F} \cap \Gamma^0_h \neq \emptyset$, where $\Gamma^0_h := \{ x\in\Gamma_h : \Omega\cdot\boldsymbol{n}_h(x)=0 \}$. Then for any $x\in\mathcal{E}^0_h$,
        \begin{equation}\label{eq3.16}
            |\Omega\cdot\boldsymbol{n}_h(x)| \lesssim h \,\|\boldsymbol{n}_h\|_{W^{1,\infty}(\mathcal{E}^0_h)},
        \end{equation}
        where the norm is understood as $\max_{F\in\mathcal{E}^0_h} \|\boldsymbol{n}_h\|_{W^{1,\infty}(F)}$. Moreover, for any $x\in \Gamma^{-}_h \Delta \Phi_h(\Gamma^{-})$,
        \begin{equation}\label{eq3.17}
            |\Omega\cdot\boldsymbol{n}_h(x)| \lesssim h^k \,\|\boldsymbol{n}_h\|_{W^{1,\infty}(\Gamma_h)}.
        \end{equation}
    \end{lemma}

    \begin{proof}\quad
        For any $x\in\mathcal{E}^0_h$, there exists a face $F$ containing $x$ with $\overline{F}\cap\Gamma^0_h \neq \emptyset$. Choose $x_0\in \overline{F}\cap\Gamma^0_h$ such that $|x-x_0|\lesssim h$. Since $\Omega\cdot\boldsymbol{n}_h$ is Lipschitz on $F$ and $\Omega\cdot\boldsymbol{n}_h(x_0)=0$, we obtain
        \begin{equation}\label{eq3.17.1}
            |\Omega\cdot \boldsymbol{n}_h(x)| = |\Omega\cdot \boldsymbol{n}_h(x)-\Omega\cdot \boldsymbol{n}_h(x_0)|
            \le h \,\|\boldsymbol{n}_h\|_{W^{1,\infty}(F)} \le h \,\|\boldsymbol{n}_h\|_{W^{1,\infty}(\mathcal{E}^0_h)}.
        \end{equation}
        This proves (\ref{eq3.16}).

        Lemma \ref{lemma3.5} implies that every $x\in \Gamma^{-}_h \Delta \Phi_h(\Gamma^{-})$ satisfies $\operatorname{dist}(x,\Gamma^0_h) \lesssim h^k$. Since $\Gamma^0_h$ is closed, there exists $x_0\in\Gamma^0_h$ with $|x-x_0|\lesssim h^k$. Applying the Lipschitz continuity of $\Omega\cdot\boldsymbol{n}_h$ on $\Gamma_h$ gives
        \begin{equation}\label{eq3.17.2}
            |\Omega\cdot \boldsymbol{n}_h(x)| = |\Omega\cdot \boldsymbol{n}_h(x)-\Omega\cdot \boldsymbol{n}_h(x_0)|
            \lesssim h^k \,\|\boldsymbol{n}_h\|_{W^{1,\infty}(\Gamma_h)}.
        \end{equation}
        Thus (\ref{eq3.17}) holds, completing the proof.
    \end{proof}

    Now we present our main results of this paper.
    \begin{theorem}\label{theorem3.1}
        Let $I$ be the solution of (\ref{eq1.1}) and let $\tilde{I}\in H^{k+1}(\mathcal{T}_h)$ denote its regular extension to the computational domain $D_h$.
        Let $f_h\in W^{1,\infty}(\mathcal{T}_h)$ and let $I_h$ be the numerical solution of (\ref{eq2.4}).
        Suppose that assumptions $(H_1)$-$(H_3)$ and $(A_1)$-$(A_2)$ hold. Then for any $k\geq 2$,
        \begin{equation}
            \|\tilde{I} - I_h\|_{DG} \lesssim h^{k+\frac{1}{2}}
            \left(
                \|f_h\|_{W^{1,\infty}(\mathcal{T}_h)} +
                \|\tilde{I}\|_{H^{k+1}(\mathcal{T}_h)} +
                \|\tilde{I}\|_{L^{\infty}(\Gamma_h)}
            \right). \label{eq3.18}
        \end{equation}
    \end{theorem}
    \begin{proof}\quad
    Applying the triangle inequality, we decompose the DG error between $\tilde{I}$ and $I_h$ into two parts
    \begin{align}\label{eq3.19}
        ||\tilde{I} - I_h||_{DG}
        &\leq ||\tilde{I} - \Lambda\tilde{I}||_{DG} + ||\Lambda\tilde{I} - I_h||_{DG}\nonumber\\
        &= ||\tilde{I} - \Lambda\tilde{I}||_{DG} + \frac{\mathcal{B}(\Lambda\tilde{I} - I_h,\Lambda\tilde{I} - I_h)}{||\Lambda\tilde{I} - I_h||_{DG}}\nonumber \text{\quad (by the definition of $||\cdot||_{DG}$)}\\
        &= ||\tilde{I} - \Lambda\tilde{I}||_{DG} + \frac{\mathcal{B}(\Lambda\tilde{I} -\tilde{I} + \tilde{I}- I_h,\Lambda\tilde{I} - I_h)}{||\Lambda\tilde{I} - I_h||_{DG}}\nonumber\\
        &\lesssim ||\tilde{I} - \Lambda\tilde{I}||_{DG} + \frac{\mathcal{B}(\tilde{I} - I_h,\Lambda\tilde{I} - I_h)}{||\Lambda\tilde{I} - I_h||_{DG}}\nonumber\quad\text{(by the Lemma \ref{lemma3.2})}\\
        &=:L_1 + L_2.
    \end{align}

    The term $L_1$ is directly bounded by Lemma \ref{lemma3.1}, so it remains to estimate $L_2$. To this end, we evaluate the bilinear form $\mathcal{B}(\cdot,\cdot)$ defined in (\ref{eq2.5}) at the extended solution $\tilde{I}$. Using the definition of $\tilde{f}$ from Section \ref{sec:3} and the continuity of $\tilde{I}$ across interior faces, we arrive at
    \begin{align}\label{eq3.20}
        & \mathcal{B}(\tilde{I},v_h)\nonumber\\
        =& \sum_{K\in\mathcal{T}_h}
        \left(
        \int_{K}(\Omega\cdot\nabla \tilde{I}+\sigma \tilde{I})v_hdx \right)
        -\sum_{K\in\mathcal{T}_h}\left(\int_{\partial_{-}K\cap\Gamma^{-}_h}(\Omega\cdot \boldsymbol{n_h})\tilde{I}v_hds \right)\nonumber\\
        &- \sum_{K\in\mathcal{T}_h}\left(\int_{\partial_{-}K \setminus \Gamma_h}(\Omega\cdot \boldsymbol{n_h}) \llbracket \tilde{I} \rrbracket v_h^{+}ds
        \right)\nonumber\\
        =& \sum_{K\in\mathcal{T}_h}
        \left(
        \int_{K}\tilde{f}v_hdx \right)
        - \sum_{K\in\mathcal{T}_h}\left(\int_{\partial_{-}K\cap\Gamma^{-}_h}(\Omega\cdot \boldsymbol{n_h})\tilde{I}v_hds \right).
    \end{align}
    Here and in the sequel, we set $v_h = \Lambda\tilde{I} - I_h$ for convenience.
    Then, recalling the discrete problem (2.4) and the property that $\tilde{f}=f$ in $D_h$, we observe
    \begin{align}\label{eq3.21}
        &\mathcal{B}(I_h,v_h)= l(v_h)\nonumber\\
        &=\sum_{K\in\mathcal{T}_h}
        \left(
        \int_{K}f_hv_hdx\right) -
        \sum_{K\in\mathcal{T}_h}
        \left(\int_{\partial_{-}K\cap\Gamma^{-}_h}(\Omega\cdot \boldsymbol{n_h})g_hv_hds
        \right)\nonumber\\
        &=\sum_{K\in\mathcal{T}_h}
        \left(
        \int_{K}f\circ \Phi_h^{-1}v_hdx\right) -
        \sum_{K\in\mathcal{T}_h}
        \left(\int_{\partial_{-}K\cap\Gamma^{-}_h}(\Omega\cdot \boldsymbol{n_h})g_hv_hds
        \right)\nonumber\\
        &=\sum_{K\in\mathcal{T}_h}
        \left(
        \int_{K}\tilde{f}\circ \Phi_h^{-1}v_hdx\right) -
        \sum_{K\in\mathcal{T}_h}
        \left(\int_{\partial_{-}K\cap\Gamma^{-}_h}(\Omega\cdot \boldsymbol{n_h})g_hv_hds
        \right).
    \end{align}

    Combining (\ref{eq3.20}) and (\ref{eq3.21}) yields
    \begin{align}\label{eq3.22}
        &\mathcal{B}(\tilde{I} - I_h, v_h)
        = \sum_{K\in\mathcal{T}_h}
        \left(
        \int_{K}(\tilde{f}-\tilde{f}\circ \Phi_h^{-1})v_hdx\right)
        \nonumber\\
        & \ +
        \sum_{K\in\mathcal{T}_h}
        \left(\int_{\partial_{-}K\cap\Gamma^{-}_h}(\Omega\cdot \boldsymbol{n_h})g_hv_hds
        - \int_{\partial_{-}K\cap\Gamma^{-}_h}(\Omega\cdot \boldsymbol{n_h})\tilde{I}v_hds\right)\nonumber\\
        &=:E_1 + E_2,
    \end{align}
    where the first term is estimated as
    \begin{align}\label{eq3.23}
        E_1 = &\sum_{K\in\mathcal{T}_h}
        \left(
        \int_{K}(\tilde{f}-\tilde{f}\circ \Phi_h^{-1})v_hdx\right)
        \leq h^{k+1}||\tilde{f}||_{W^{1, \infty}(\mathcal{T}_h)}||v_h||_{DG}.
    \end{align}

    For $E_2$, observe that on $\Phi_h(\Gamma^-)$ we have
    \begin{align*}
      g_h = \tilde{g}\circ\Phi_h^{-1} = g\circ\Phi_h^{-1} = I\circ\Phi_h^{-1} = I^\sharp, \ \ on \ \ \Phi_h(\Gamma^-).
    \end{align*}
    This allows us to split $E_2$ into three contributions according to the location on $\Gamma_h^-$
    \begin{align}\label{eq3.24}
        E_2=& \sum_{K\in\mathcal{T}_h}
        \left(\int_{\partial_{-}K\cap\Gamma^{-}_h}(\Omega\cdot \boldsymbol{n_h})g_hv_hds
        - \int_{\partial_{-}K\cap\Gamma^{-}_h}(\Omega\cdot \boldsymbol{n_h})\tilde{I}v_hds\right)\nonumber\\
        =& \sum_{K\in\mathcal{T}_h}
        \left(\int_{\partial_{-}K\cap\left(\Gamma^{-}_h\cap\Phi_h(\Gamma^-)\right)}(\Omega\cdot \boldsymbol{n_h})(I^\sharp - \tilde{I}) v_hds\right)\nonumber\\
        &+\sum_{K\in\mathcal{T}_h}\left(\int_{\partial_{-}K\cap\left(\Gamma^{-}_h \setminus \Phi_h(\Gamma^-)\right)}(\Omega\cdot \boldsymbol{n_h})g_hv_hds\right)\nonumber\\
        &-\sum_{K\in\mathcal{T}_h}
        \left(\int_{\partial_{-}K\cap\left(\Gamma^{-}_h \setminus \Phi_h(\Gamma^-)\right)}(\Omega\cdot \boldsymbol{n_h})\tilde{I}v_hds\right)\nonumber\\
        =&: T_1 + T_2 + T_3.
    \end{align}

    The term $T_1$ is directly controlled by Lemma \ref{lemma3.3}, which gives $T_1 \lesssim h^{k+1/2}\|\tilde{I}\|_{L^2(\mathcal{T}_h)}$.
    For $T_2$, recall that $g_h = \tilde{g}\circ\Phi_h^{-1}$ and $\tilde{g}$ is the zero extension of $g$ on $\Gamma$.  Hence $T_2 = 0$.

    It remains to bound $T_3$.  Let $D_c$ denote the collection of boundary elements that intersect the characteristic boundary $\Gamma_h^0$, i.e.,
    \[
    D_c = \bigl\{ K\in\mathcal{T}_h : \partial K\cap \Gamma_h^0 \neq \emptyset \bigr\}.
    \]
    Using the estimate (\ref{eq2.3}), the Cauchy-Schwarz inequality, and Lemma \ref{lemma3.6}, we obtain
    \begin{align}\label{eq3.25}
        T_3 &=
        -\sum_{K\in\mathcal{T}_h}
        \left(\int_{\partial_{-}K\cap S}(\Omega\cdot \boldsymbol{n_h})\tilde{I}v_hds\right)\nonumber\\
        &\lesssim h^k\sum_{K\in\mathcal{T}_h}||\tilde{I}||_{L^2\left(\partial_{-}K\cap S\right)}
         ||v_h||_{L^{2}\left(\partial_{-}K\cap S\right)}\nonumber\\
        &\lesssim h^k
        \left(\sum_{K\in\mathcal{T}_h}
        ||\tilde{I}||^2_{L^2\left(\partial_{-}K\cap S\right)}
        \right)^{\frac{1}{2}}
        \left(
        \sum_{K\in\mathcal{T}_h}||v_h||^2_{L^{2}\left(\partial_{-}K\cap S\right)}
        \right)^{\frac{1}{2}}\nonumber\\
        &\lesssim h^k
        ||\tilde{I}||_{L^\infty(\Gamma_h)}Area(\Gamma_h^-\Delta\Phi_h(\Gamma^-))
        \left(\sum_{K\in D_c}
        \left\{h_K^{-1}||v_h||_{L^2(K)}^2 + h_K||v_h||_{H^1(K)}^2\right\}
        \right)^{\frac{1}{2}}\nonumber\\
        &\lesssim h^{k}\cdot h^{\frac{k-1}{2}}||\tilde{I}||_{L^\infty(\Gamma_h)}||v_h||_{DG}
        \lesssim h^{k+\frac{1}{2}}||\tilde{I}||_{L^\infty(\Gamma_h)}||v_h||_{DG}
    \end{align}

    Here $S := \Gamma^{-}_h\setminus\Phi_h(\Gamma^-)$ for brevity, and we have applied the trace inequality together with the inverse estimate to convert the boundary integrals into elementwise $L^2$ and $H^1$ norms.

    Finally, gathering the estimates (\ref{eq3.1}), (\ref{eq3.8}), (\ref{eq3.23}) and (\ref{eq3.25}), we arrive at the desired bound (\ref{eq3.18}).  This completes the proof of Theorem \ref{theorem3.1}.
    \end{proof}

\section{Numerical experiments}\label{sec:experiments}
    In this section, we carry out numerical experiments on curved domains to verify the theoretical results. These experiments are conducted by employing the open-source finite element software package NUDG++ \cite{HESTHAVEN2008} for 2D case and the adaptive finite element package Parallel Hierarchical Grid (PHG) \cite{ZHANG2009} for 3D case.

    \subsection{Example two-dimensional(2D)}

    To construct a representative test case, we make the following specific choices. For the RTE with the absorption coefficient $\sigma = 1$, the discrete direction is set to the unit vector $\Omega = (\frac{\sqrt{3}}{2}, \frac{1}{2})^T$, the exact solution is taken as $I(x,y) = \sin(\pi x+\pi y) + x^2 + y^2 + xy + 5$, and the original domain is the curved disc defined by

    \begin{align*}
    D = \{(x, y): x^2 + y^2 < 0.25\}.
    \end{align*}

    The 2D straight mesh $\mathcal{T}_{h_1}\sim \mathcal{T}_{h_6}$ are generated by Gmsh with global mesh size factor $2, 1, 0.5, 0.25, 0.125$ and $0.0625$, receptively. Moreover, utilizing the approximation $h_i\varpropto Ndof_{i}^{-\frac{1}{2}} \ (i=1,2,...,6)$, where $Ndof_i$ is the number of Degree of Freedom(DOF), we could estimate the error convergence rate by

    \begin{equation}\label{eq4.1}
      rate_i = \frac{\ln(error_i/error_{i-1})}{\ln(h_i/h_{i-1})}, \quad i = 2, ..., 6.
    \end{equation}

    \begin{figure}[b]
    \centering
    {
        \includegraphics[width=0.4\textwidth]{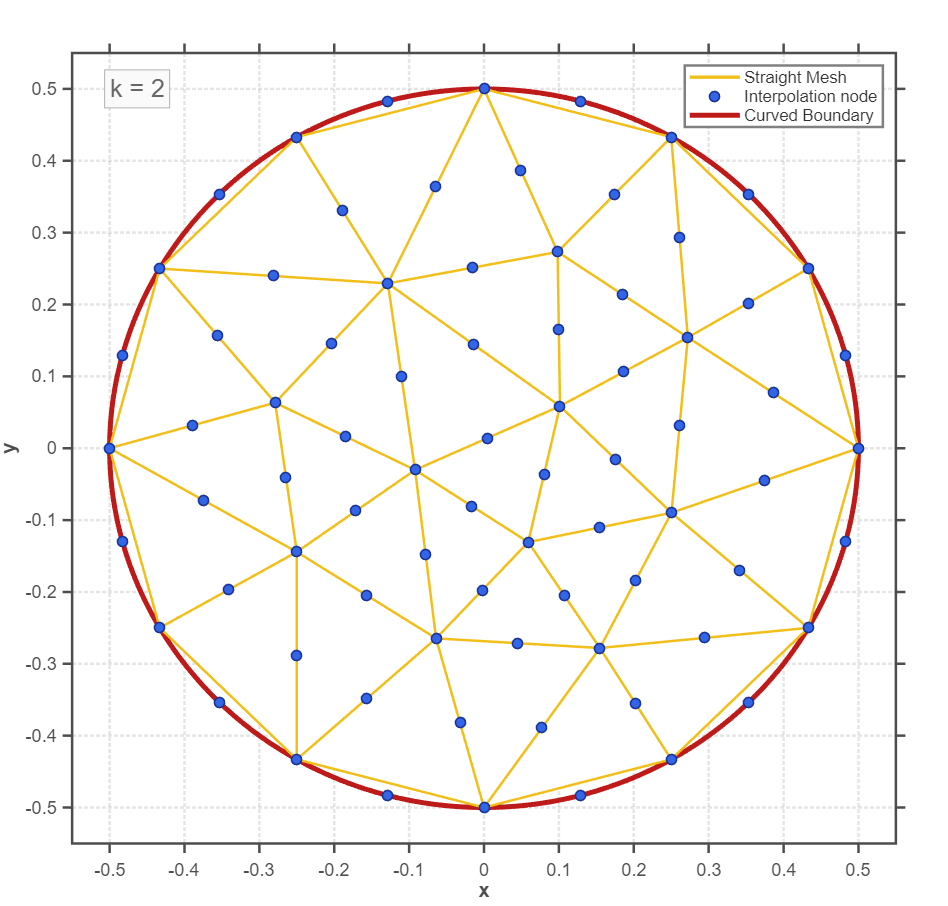}
    }
    \qquad
    \label{fig:N=2(order)}
    {
        \includegraphics[width=0.4\textwidth]{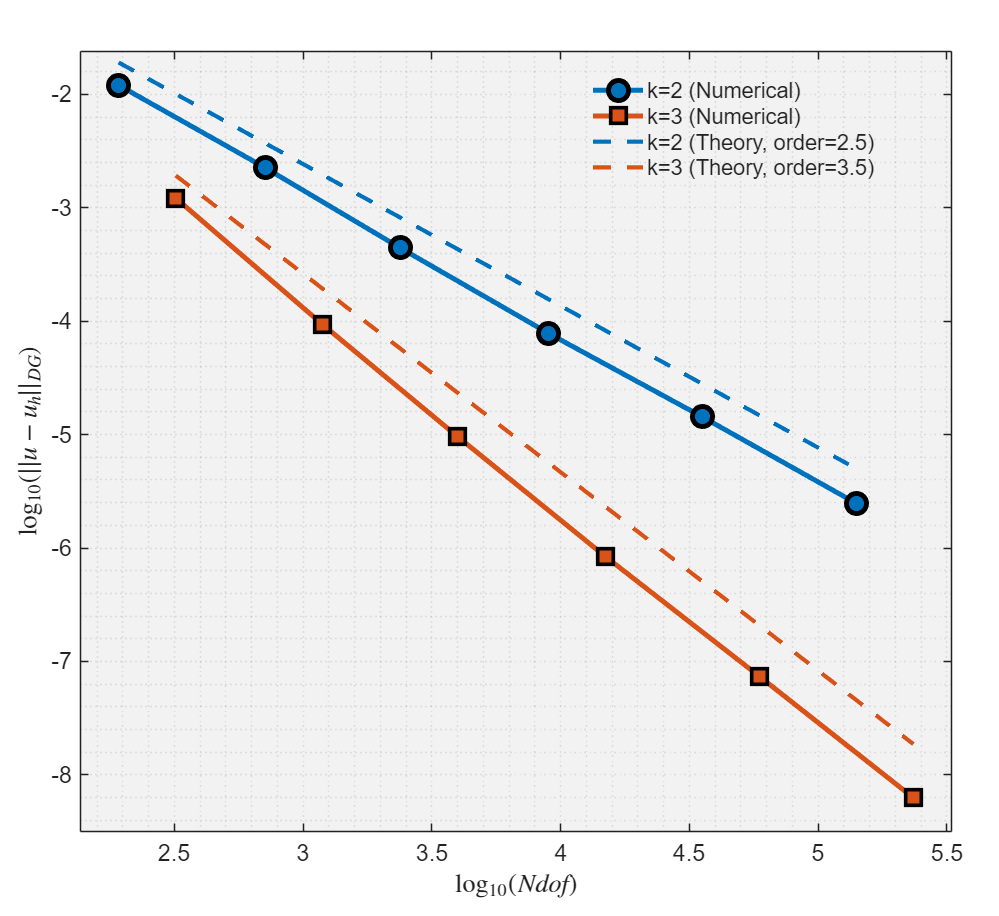}
    }
    \caption{Isoparametric mapping and numerical results in 2D.}
    \end{figure}

    \begin{table}[t]
    \centering
    \caption{Convergence rate on mesh $\mathcal{T}_h$ of unit circle with $k=2$ \& $k=3$}
    \begin{tabular}{llllllll}
    \hline\noalign{\smallskip}
    $k$ & Mesh & Nelem & Ndof & $L^2$ error & $L^2$ rate & DG error & DG rate \\
    \hline\noalign{\smallskip}
    \multirow{6}{*}{2} & $\mathcal{T}_{h_1}$ & 32 & 192 & 3.3747e-03 & - & 1.2060e-02 & - \\\noalign{\smallskip}
    \noalign{\smallskip}
    & $\mathcal{T}_{h_2}$ & 119 & 714 & 4.5822e-04 & 3.0405 & 2.2467e-03 & 2.5589 \\\noalign{\smallskip}
    \noalign{\smallskip}
    & $\mathcal{T}_{h_3}$ & 397 & 2382 & 6.1777e-05 & 3.3264 & 4.4307e-04 & 2.6950 \\\noalign{\smallskip}
    \noalign{\smallskip}
    & $\mathcal{T}_{h_4}$ & 1497 & 8982 & 7.9332e-06 & 3.0927 & 7.8135e-05 & 2.6148 \\\noalign{\smallskip}
    \noalign{\smallskip}
    & $\mathcal{T}_{h_5}$ & 5908 & 35448 & 1.0275e-06 & 2.9776 & 1.4367e-05 & 2.4672 \\\noalign{\smallskip}
    \noalign{\smallskip}
    & $\mathcal{T}_{h_6}$ & 23452 & 140712 & 1.3519e-07 & 2.9424 & 2.4791e-06 & 2.5489 \\
    \hline\noalign{\smallskip}
    \multirow{6}{*}{3} & $\mathcal{T}_{h_1}$ & 32 & 320 & 2.7108e-04 & - & 1.2164e-03 & - \\\noalign{\smallskip}
    \noalign{\smallskip}
    & $\mathcal{T}_{h_2}$ & 119 & 1190 & 1.5649e-05 & 4.3429 & 9.3544e-05 & 3.9062 \\\noalign{\smallskip}
    \noalign{\smallskip}
    & $\mathcal{T}_{h_3}$ & 397 & 3970 & 1.1180e-06 & 4.3806 & 9.6579e-06 & 3.7693 \\\noalign{\smallskip}
    \noalign{\smallskip}
    & $\mathcal{T}_{h_4}$ & 1497 & 14970 & 7.0975e-08 & 4.1543 & 8.3750e-07 & 3.6844 \\\noalign{\smallskip}
    \noalign{\smallskip}
    & $\mathcal{T}_{h_5}$ & 5908 & 59080 & 4.5273e-09 & 4.0095 & 7.4028e-08 & 3.5342 \\\noalign{\smallskip}
    \noalign{\smallskip}
    & $\mathcal{T}_{h_6}$ & 23452 & 234520 & 2.9106e-10 & 3.9812 & 6.2822e-09 & 3.5785 \\
    \hline
    \end{tabular}\label{table1}
    \end{table}

    As shown in Table \ref{table1}, the error convergence rates under the DG norms agree perfectly with the Theorem \ref{theorem3.1} for $k=2$ and $k=3$.

    \begin{figure}[!htbp]
        \centering
        \label{fig:straight_convex}
        {
            \includegraphics[width=0.3\textwidth]{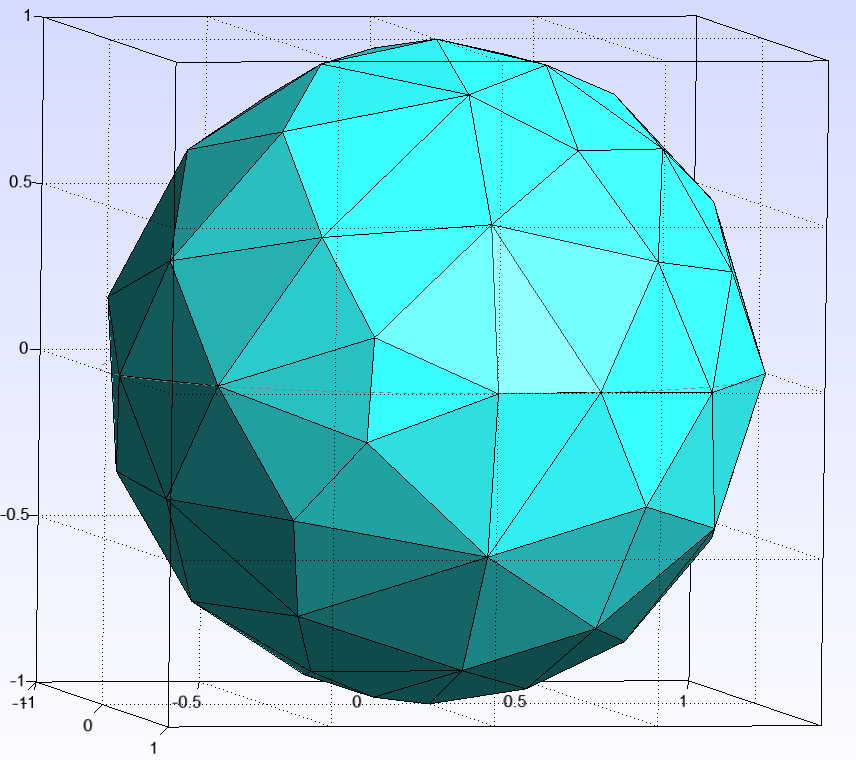}
        }
        \hfill
        \label{fig:straight_curved}
        {
            \includegraphics[width=0.3\textwidth]{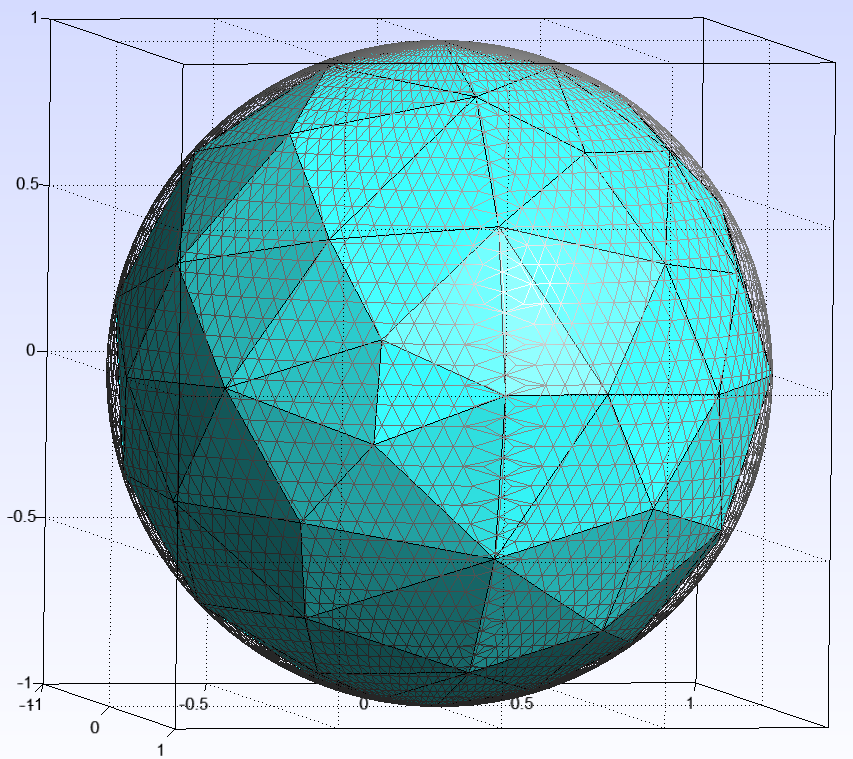}
        }
        \hfill
        \label{fig:curved_mesh}
        {
            \includegraphics[width=0.3\textwidth]{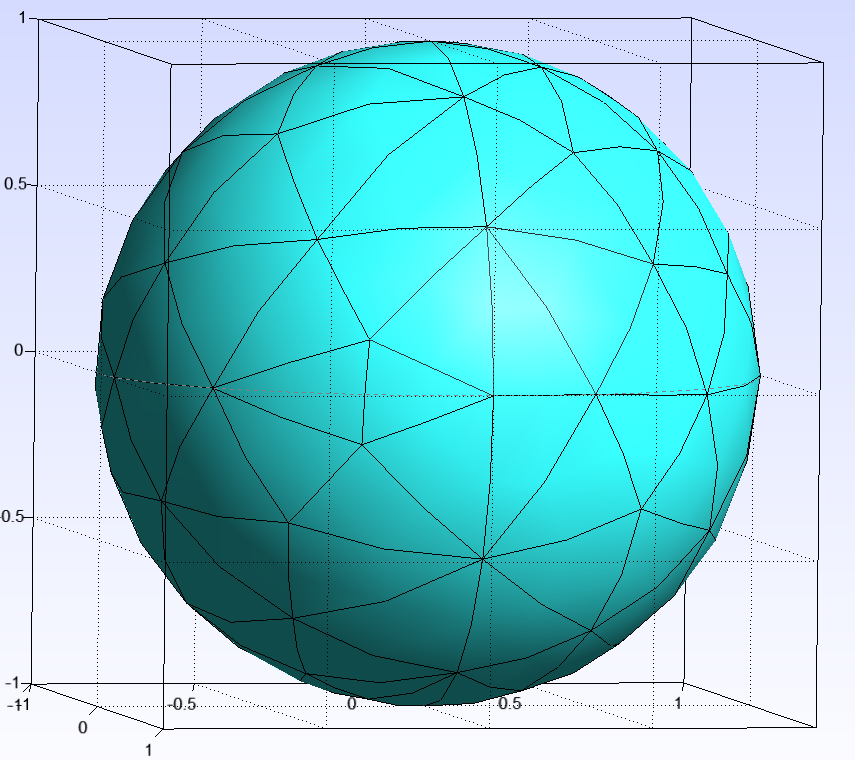}
        }
        \label{fig:second_figures}
    \caption{Mesh and domain evolution through isoparametric mapping in 3D.}
    \end{figure}

    \subsection{Example three-dimensional(3D)}
    For 3D case, we conduct a series of comparative experiments on $\widehat{\mathcal{T}}_h$ and $\mathcal{T}_h$ to demonstrate the advantages of computation on curved domains. For brevity, we select the absorption coefficient $\sigma = 1$, the discrete direction $\Omega = \left( \frac{1}{\sqrt{3}}, \frac{1}{\sqrt{3}}, \frac{1}{\sqrt{3}} \right)^T$, the exact solution $I(x, y, z) = \sin(\pi x + \pi y + \pi z) + x^2 + y^2 + z^2 + x y z + 5$ and the original domain is taken as the curved region defined by
    \begin{align*}
        D = \{(x, y, z): x^2 + y^2 + z^2 <1\}.
    \end{align*}

    In the actual calculation, we transform the integrals of curved elements $K\in\mathcal{T}_h$ and surface integrals to those of straight elements $\widehat{K}\in\mathcal{\widehat{T}}_h$ for calculation \cite{LI2025}
    \begin{equation}\label{eq4.2}
      \int_K fd\boldsymbol{x} = \int_{\widehat{K}} \hat{f}\mathcal{J}_{F_h(x)}d\hat{\boldsymbol{x}},\quad
      \int_F fd\boldsymbol{s} = \int_{\widehat{F}} \hat{f}\mathcal{J}_{F_h(\hat{x})}||\mathbb{J}_{F_h(\hat{x})}^{-T}\hat{\boldsymbol{n}}||d\hat{\boldsymbol{s}}.
    \end{equation}
    where $\mathbb{J}_{F_h(\hat{x})}$ and $\mathcal{J}_{F_h(x)}$ are the Jacobian matrix and its determinant of $F_h$, $\hat{\boldsymbol{n}}$ is the unit outer normal of $\widehat{F}$.
    Moreover, employing the approximation $h_i\varpropto N_{i}^{-\frac{1}{3}} \ (i=1,2,...,6)$, we also use the formula (\ref{eq4.1}) to compute the convergence rate.

    \begin{table}[!htbp]
    \centering
    \caption{Convergence rate on straight mesh $\mathcal{\widehat{T}}_h$ of convex polyhedron domain $\widehat{D}_h$ with $k=2$}
    \begin{tabular}{l l l l l l l}
    \hline\noalign{\smallskip}
    Mesh & Nelem & Ndof & $L^2$ error & $L^2$ rate & DG error & DG rate \\\noalign{\smallskip}
    \hline\noalign{\smallskip}
    $\mathcal{\widehat{T}}_{h_1}$ & 48 & 480 & 3.4686e-01 & - & 5.0478e-01 & - \\\noalign{\smallskip}
    \hline\noalign{\smallskip}
    $\mathcal{\widehat{T}}_{h_2}$ & 384 & 3840 & 1.2340e-01 & 1.4910 & 2.0939e-01 & 1.2695 \\\noalign{\smallskip}
    \hline\noalign{\smallskip}
    $\mathcal{\widehat{T}}_{h_3}$ & 3072 & 30720 & 1.5694e-02 & 2.9751 & 4.7381e-02 & 2.1438 \\\noalign{\smallskip}
    \hline\noalign{\smallskip}
    $\mathcal{\widehat{T}}_{h_4}$ & 24576 & 245760 & 1.6289e-03 & 3.2682 & 8.5384e-03 & 2.4723 \\\noalign{\smallskip}
    \hline\noalign{\smallskip}
    $\mathcal{\widehat{T}}_{h_5}$ & 196608 & 1966080 & 1.9103e-04 & 3.0920 & 1.5069e-03 & 2.5024 \\\noalign{\smallskip}
    \hline
    \end{tabular}\label{table2}
    \end{table}
    The straight mesh $\widehat{\mathcal{T}}_{h_1}$ in Table {\ref{table2}} and Table {\ref{table3}} is generated by Netgen, $\widehat{\mathcal{T}}_{h_2}\sim \widehat{\mathcal{T}}_{h_5}$ are derived from the command phgRefineAllElements with $\widehat{\mathcal{T}}_{h_1}$. The curved mesh $\mathcal{T}_{h_i}$ in Table {\ref{table4}} is derived from $\widehat{\mathcal{T}}_{h_i}(i=1,2,...,6)$ by the isoparametric mapping $F_h$.
    \begin{table}[!htbp]
    \centering
    \caption{Convergence rate on straight mesh $\mathcal{\widehat{T}}_h$ of curved domain $D_h$ with $k=2$}
    \begin{tabular}{l l l l l l l}
    \hline\noalign{\smallskip}
    Mesh & Nelem & Ndof & $L^2$ error & $L^2$ rate & DG error & DG rate \\\noalign{\smallskip}
    \hline\noalign{\smallskip}
    $\mathcal{\widehat{T}}_{h_1}$ & 48 & 480 & 4.8805e-01 & - & 7.0139e-01 & - \\\noalign{\smallskip}
    \hline\noalign{\smallskip}
    $\mathcal{\widehat{T}}_{h_2}$ & 384 & 3840 & 1.6281e-01 & 1.5838 & 2.7039e-01 & 1.3752 \\\noalign{\smallskip}
    \hline\noalign{\smallskip}
    $\mathcal{\widehat{T}}_{h_3}$ & 3072 & 30720 & 2.8623e-02 & 2.5079 & 6.2308e-02 & 2.1176 \\\noalign{\smallskip}
    \hline\noalign{\smallskip}
    $\mathcal{\widehat{T}}_{h_4}$ & 24576 & 245760 & 6.2782e-03 & 2.1888 & 1.3013e-02 & 2.2595 \\\noalign{\smallskip}
    \hline\noalign{\smallskip}
    $\mathcal{\widehat{T}}_{h_5}$ & 196608 & 1966080 & 1.5401e-03 & 2.0273 & 2.8423e-03 & 2.1948 \\\noalign{\smallskip}
    \hline
    \end{tabular}\label{table3}
    \end{table}

    We compare the convergence rates for different mesh-domain configurations in Table \ref{table2} and Table \ref{table3}. While a straight mesh $\mathcal{\widehat{T}}_{h}$ on a convex polyhedron domain $\widehat{D}_h$ yields optimal convergence in both the $L^2$ and DG norms \cite{BRENNER2008,CIARLET1972,PAZNER2021}, applying the same mesh to corresponding curved domain $D_h$ results in degraded performance. The inaccuracy of the straight mesh in capturing the boundary geometry leads to the dominance of the geometric approximation error \cite{CIARLET1972}.
    \begin{table}[!htbp]
    \centering
    \caption{Convergence rate on curved mesh $\mathcal{T}_h$ of curved domain $D_h$ with $k=2$}
    \begin{tabular}{l l l l l l l}
    \hline\noalign{\smallskip}
    Mesh & Nelem & Ndof & $L^2$ error & $L^2$ rate & DG error & DG rate \\\noalign{\smallskip}
    \hline\noalign{\smallskip}
    $\mathcal{T}_{h_1}$ & 48 & 480 & 2.5781e-01 & - & 7.3825e-01 & - \\\noalign{\smallskip}
    \hline\noalign{\smallskip}
    $\mathcal{T}_{h_2}$ & 384 & 3840 & 6.4164e-02 & 2.0065 & 2.0012e-01 & 1.8832 \\\noalign{\smallskip}
    \hline\noalign{\smallskip}
    $\mathcal{T}_{h_3}$ & 3072 & 30720 & 1.2140e-02 & 2.4020 & 4.5983e-02 & 2.1217 \\\noalign{\smallskip}
    \hline\noalign{\smallskip}
    $\mathcal{T}_{h_4}$ & 24576 & 245760 & 1.5993e-03 & 2.9243 & 8.5462e-03 & 2.4277 \\\noalign{\smallskip}
    \hline\noalign{\smallskip}
    $\mathcal{T}_{h_5}$ & 196608 & 1966080 & 1.9688e-04 & 3.0221 & 1.5123e-03 & 2.4985 \\\noalign{\smallskip}
    \hline
    \end{tabular}\label{table4}
    \end{table}

    The results in Table \ref{table4} illustrate the optimal convergence rate in the DG norm achieved by a curved mesh $\mathcal{T}_h$ on a curved domain $D_h$, which confirms the Theorem \ref{theorem3.1}. This success demonstrates that isoparametric elements by high-order methods are essential for achieving the optimal accuracy on curved domains.

\section{Conclusion}
We have presented a rigorous error analysis for the isoparametric upwind DG method applied to the RTE on curved domains. By estimating the approximation error of the isoparametric finite element space and the consistency error of the upwind DG scheme, we derived an optimal convergence rate of $\mathcal{O}(h^{k+\frac{1}{2}})$ in the DG norm. This result balances the geometric approximation error and the numerical discretization error. Two-dimensional numerical experiments confirmed the theoretical predictions for $k=2$ and $k=3$. Furthermore, three-dimensional examples demonstrated the advantages of isoparametric elements for high-order methods on curved domains. The analysis and numerical results establish the efficacy of the proposed approach for problems involving curved geometries.

\end{document}